\documentclass[12pt]{amsart}

\usepackage{aliascnt,amsmath,amssymb,amsthm,amsfonts,enumerate,mathrsfs,latexsym,mathtools,bm,xcolor,shuffle}
\usepackage{stmaryrd}	
\usepackage[abbrev]{amsrefs}
\usepackage{appendix}

\usepackage[OT2,T1]{fontenc}
\usepackage{scalerel}

\usepackage[marginparwidth=0pt,margin=24truemm]{geometry}

\definecolor{mylinkcolor}{RGB}{20, 160, 80}
\definecolor{mycitecolor}{RGB}{20, 80, 140}

\usepackage{hyperref}
\usepackage[nameinlink]{cleveref}
\hypersetup{setpagesize=false, bookmarksnumbered=true, bookmarksopen=true, colorlinks=true, linkcolor=mylinkcolor, citecolor=mycitecolor,}
\usepackage{autonum}

\usepackage{tikz}
\usetikzlibrary{intersections,calc,arrows.meta,decorations,decorations.pathreplacing}

\theoremstyle{plain}

\newtheorem{thm}{Theorem}[section]
\crefname{thm}{Theorem}{Theorems}

\newaliascnt{lem}{thm} 
\newtheorem{lem}[lem]{Lemma}
\aliascntresetthe{lem}
\crefname{lem}{Lemma}{Lemmas}

\newaliascnt{prop}{thm}
\newtheorem{prop}[prop]{Proposition}
\aliascntresetthe{prop}
\crefname{prop}{Proposition}{Propositions}

\newaliascnt{cor}{thm}
\newtheorem{cor}[cor]{Corollary}
\aliascntresetthe{cor}
\crefname{cor}{Corollary}{Corollaries}

\newaliascnt{fact}{thm}

\aliascntresetthe{fact}
\crefname{fact}{Fact}{Facts}

\newaliascnt{conj}{thm}
\newtheorem{conj}[conj]{Conjecture}
\aliascntresetthe{conj}
\crefname{conj}{Conjecture}{Conjectures}

\newaliascnt{dfn}{thm}
\newtheorem{dfn}[dfn]{Definition}
\aliascntresetthe{dfn}
\crefname{dfn}{Definition}{Definitions}

\newaliascnt{rem}{thm}
\newtheorem{rem}[rem]{Remark}
\aliascntresetthe{rem}
\crefname{rem}{Remark}{Remarks}

\newaliascnt{ex}{thm}

\aliascntresetthe{ex}
\crefname{ex}{Example}{Examples}

\newtheorem{mainthm}{Main Theorem}
 
\newaliascnt{mainconj}{mainthm}   
 
\aliascntresetthe{mainconj}        
\crefname{mainconj}{Main Conjecture}{Main Conjectures}

\allowdisplaybreaks[2]

\everymath{\displaystyle}

\newcommand{\QQ}{\mathbb{Q}}
\newcommand{\ZZ}{\mathbb{Z}}

\newcommand{\CC}{\mathbb{C}}

\DeclareMathOperator{\Qhat}{\widehat{\mathcal{Q}}}
\DeclareMathOperator{\Ahat}{\widehat{\mathcal{A}}}
\DeclareMathOperator{\Shat}{\widehat{\mathcal{S}}}

\DeclareMathOperator{\Span}{Span}

\newcommand\quotient[2]{
	\mathchoice
	{
		\text{\raise1ex\hbox{$#1$}\Big/\lower1ex\hbox{$#2$}}%
	}
	{
		#1\,/\,#2
	}
	{
		#1\,/\,#2
	}
	{
		#1\,/\,#2
	}
}

\title{Duality formulas for $\Qhat$-multiple zeta values}

\author{Koki Ishida}
\address{Mathematical Institute, Tohoku University, Sendai, Japan.}
\email{ishida.koki.p1@dc.tohoku.ac.jp}

\date{}

\subjclass[2020]{11M32, 11R18, 05A30}

\keywords{multiple zeta value, $\bm{p}$-adic multiple zeta value, $t$-adic multiple zeta value, $q$-analogue of multiple zeta value, multiple harmonic $q$-sum}

\begin{document}

\begin{abstract}
In this paper, we introduce two duality formulas for $\Qhat$-multiple zeta values ($\Qhat$-MZVs for short) which are defined by Takeyama and Tasaka.
Since taking two different limits of $\Qhat$-MZV recovers the $\bm{p}$-adic multiple zeta value and the $t$-adic multiple zeta value respectively, finding the relation of $\Qhat$-MZVs partially supports the Kaneko--Zagier conjecture and its refined version.
Currently, three types of duality formulas for the $\bm{p}$-adic multiple zeta values are known, and for one of them, the $q$-analogue was studied by Takeyama and Tasaka.
We present $q$-analogues of the remaining two duality formulas for the $\bm{p}$-adic multiple zeta values.
\end{abstract}

\maketitle

\section{Introduction}
The Kaneko--Zagier conjecture, in \cite{KZ2}, states that there is a one-to-one correspondence between the finite multiple zeta value $\zeta_{\mathcal{A}}(\bm{k}) \in \mathcal{A} = (\textstyle{\prod}_{p}\mathbb{F}_{p}) / (\textstyle{\bigoplus}_{p}\mathbb{F}_p)$ and the symmetric multiple zeta value $\zeta_{S}(\bm{k}) \in \mathcal{Z}/\pi^2\mathcal{Z}$ defined for an index $\bm{k}$.
Here, for a positive integer $r$ and a tuple of non-negative integers $\bm{k} = (k_1,\dots, k_r)$, we set $\mathrm{wt}(k) = k_1 +\cdots+ k_r$ and $\mathrm{dep}(k) = r$, and call them the weight and the depth of $\bm{k}$, respectively.
We call such $\bm{k}$ an $\emph{index}$ if none of its entries is zero.
The \emph{empty index} $\varnothing$ is allowed with $\mathrm{wt}(\varnothing) = \mathrm{dep}(\varnothing) = 0$.
For any function $F$ on the indices, we define $F(\varnothing) = 1$.
The $\QQ$-vector space $\mathcal{Z}$ is spanned by all multiple zeta values 
\[
    \zeta(\bm{k}) = \sum_{0 < m_1 < \cdots < m_r}\prod_{j = 1}^{r}\frac{1}{{m_j}^{k_j}}
\]
for indices with $k_r \geq 2$.

Furthermore, its refined version, which describes the correspondence of the $\bm{p}$-adic finite multiple zeta value $\zeta_{\Ahat}(\bm{k})$ (or $\Ahat$-MZV) and the $t$-adic symmetric multiple zeta value $\zeta_{\Shat}(\bm{k})$ (or $\Shat$-MZV), which are introduced by Rosen \cite{Ros2}, Jarossay \cite{Jar}, respectively.
According to the refined version of Kaneko--Zagier conjecture \cref{conj:KZ}, $\Ahat$-MZVs and $\Shat$-MZVs are expected to satisfy exactly the same relation.

As examples of the $\QQ$-linear relations for $\Ahat$-MZVs, the following three duality formulas are known.
\begin{thm}[\cite{Seki}*{Theorem 1.3}]\label{prop:seki}
    For a non-empty index $\bm{k}$, it holds that
    \[
        \sum_{l \geq 0} \zeta_{\Ahat}^{\star}(\bm{k}, \{1\}^{l}) \bm{p}^l =  -\sum_{l \geq 0} \zeta_{\Ahat}^{\star}(\bm{k}^{\lor}, \{1\}^{l}) \bm{p}^l.
    \]
\end{thm}
\begin{thm}[\cite{Ros1}*{Theorem 4.5}]\label{prop:rosen}
    For an index $\bm{k}$, it holds that
        \[
        \zeta_{\Ahat}(\bm{k}) + \sum_{l \geq 0}(-1)^{l} \zeta_{\Ahat}(\bm{k}\ast \{1\}^l, 1)\bm{p}^{l + 1} = (-1)^{\mathrm{dep}\,\bm{k}} \sum_{\bm{k} \preceq \bm{l}} \zeta_{\Ahat}(\bm{l}).
        \]
\end{thm}
\begin{thm}[\cite{MSW}*{Theorem 5.3}]\label{prop:msw}
    For an index $\bm{k}$, it holds that
        \[
        \zeta_{\Ahat}(\bm{k}) = (-1)^{\mathrm{dep}(\bm{k})} \sum_{l \geq 0} \left(\sum_{\begin{subarray}{c} \bm{l} \in (\mathbb{Z}_{\geq 0})^{\mathrm{dep}(\bm{k})} \\ \mathrm{wt}\, \bm{l} = l \end{subarray}}\sum_{\bm{l} \oplus \bm{k} \preceq \bm{m} \preceq \bm{l} \oslash \bm{k}} \zeta_{\Ahat}(\bm{m}) \right)\bm{p}^{l}.
        \]
\end{thm}
Since these formulas can be regarded as refinements of Hoffman's duality formula for finite multiple zeta values, we call them duality formulas as well.
\begin{thm}[\cite{Hof}*{Theorem 4.6, Theorem 4.7}]
    For a non-empty index $\bm{k}$, we have
    \begin{enumerate}
        \item[(1)] $\zeta_{\mathcal{A}}^{\star}(\bm{k}) = -\zeta_{\mathcal{A}}^{\star}(\bm{k}^{\lor})$,\vspace{0.2cm}
        \item[(2)] $\zeta_{\mathcal{A}}(\bm{k}) = (-1)^{\mathrm{dep}(k)}\sum_{\bm{k} \preceq \bm{l}} \zeta_{\mathcal{A}}(\bm{l})$.
    \end{enumerate}
\end{thm}
In \cite{BTT}, Bachmann, Takeyama and Tasaka obtained the finite multiple zeta values and the symmetric multiple zeta values by taking the "algebraic limit" and the "analytic limit" of certain multiple harmonic $q$-sums, respectively.
Inspired by this work, Takeyama and Tasaka constructed $\Qhat$-MZVs, which recover $\Ahat$-MZVs and $\Shat$-MZVs by taking the "algebraic limit" and the "analytic limit".
Since a $\QQ[q]$-linear relation of $\Qhat$-multiple zeta values implies the $\QQ$-linear relations for $\Ahat$-MZVs and $\Shat$-MZVs simultaneously, this philosophy supports the refined Kaneko--Zagier conjecture.
In \cite{TT}, the following formula is shown which is the $q$-analogue of \cref{prop:seki}.
\begin{thm}[\cite{TT}*{Theorem 5.4}]\label{prop:tt}
    For a non-empty index $\bm{k}$, it holds that
    \[
        q^{\frac{\bm{p}(\bm{p}+1)}{2}}\sum_{l \geq 0}\zeta_{\Qhat}^{BZ, \star} (\bm{k}, \{1\}^l)[\bm{p}]^l = -\sum_{l \geq 0}\overline{\zeta}_{\Qhat}^{\star}(\bm{k}^{\lor}, \{1\}^l)(q^{-\bm{p}}[\bm{p}])^l,
    \]
    where
    \[q^{\frac{\bm{p}(\bm{p} + 1)}{2}} := \left(\left( \sum_{l = 0}^{n - 1}\binom{\frac{p+1}{2}}{l}(-(1-q)[p])^l \ \bmod [p]^n\right)_p\right)_n.\]
\end{thm}
The main results of this paper are the $q$-analogues of the other two duality formulas.
These are restated as \cref{thm:qrosen} and \cref{thm:qmsw}.

\begin{mainthm}[the $q$-analogue of \cref{prop:rosen}]\label{manthmA}
    For an index $\bm{k}$, it holds that
    \begin{align}
        q^{\frac{\bm{p}(\bm{p} - 1)}{2}}\zeta_{\Qhat}^{BZ}(\bm{k}) + \sum_{l \geq 0}(-1)^{l} \zeta_{\Qhat}^{BZ}(\bm{k} \ast_q \{1\}^l, 1)[\bm{p}]^{l+1} = (-1)^{\mathrm{dep}(\bm{k})} \sum_{\bm{k} \preceq \bm{l}} \zeta_{\Qhat}^{SZ}(\bm{l}).
    \end{align}
\end{mainthm}
\begin{mainthm}[the $q$-analogue of \cref{prop:msw}]\label{mainthmB}
    For an index $\bm{k}$, it holds that
    \[
    \zeta_{\Qhat}^{BZ}(\bm{k}) = (-q^{-\bm{p}})^{\mathrm{dep}( \bm{k})} \sum_{l \geq 0} \left(\sum_{\begin{subarray}{c} \bm{l} \in (\mathbb{Z}_{\geq 0})^{\mathrm{dep}( \bm{k})} \\ \mathrm{wt}(\bm{l}) = l \end{subarray}}\sum_{\bm{l} \oplus \bm{k} \preceq \bm{m} \preceq \bm{l} \oslash \bm{k}} \zeta_{\Qhat}^{SZ}(\bm{m}) \right)([\bm{p}]q^{-\bm{p}})^{l}.
    \]
\end{mainthm}
Taking the "analytic limit" of both sides of \cref{mainthmB}, we can derive a following duality formula for $\Shat$-MZV that is a counterpart to \cref{prop:msw}, which was previously unknown in the literature.
\begin{mainthm}\label{cor:Smsw}
     For an index $\bm{k} = (k_1, \cdots , k_r)$, we have
     \[
        \zeta_{\Shat}(\bm{k}) = (-1)^{\mathrm{dep}(\bm{k})} \sum_{l \geq 0} \left(\sum_{\begin{subarray}{c} \bm{l} \in (\mathbb{Z}_{\geq 0})^{\mathrm{dep}(\bm{k})} \\ \mathrm{wt}(\bm{l}) = l \end{subarray}}\sum_{\bm{l} \oplus \bm{k} \preceq \bm{m} \preceq \bm{l} \oslash \bm{k}} \zeta_{\Shat}(\bm{m}) \right)t^{l}.
     \]
\end{mainthm}

\section*{Acknowledgments}
The author would like to express his sincere gratitude to Professor Yasuo Ohno for continuous support and helpful discussions.
The author would like to thank Professor Koji Tasaka for carefully reading an earlier draft of this paper and for valuable comments.
The author also would like to thank Professor Shin-ichiro Seki and Hanamichi Kawamura for information and helpful comments on $t$-adic symmetric multiple zeta values.

\section{Preliminaries}
In this section, we review the definition of $\Ahat,\Shat$-MZVs and collect the notation commonly used in the study of multiple zeta values and used throughout this paper.

\subsection{$\Ahat$-MZV and $\Shat$-MZV}
In this subsection, we briefly review the definition of $\Ahat$-MZV and $\Shat$-MZV.
For $n \geq 1$, we define a $\QQ$-algebra
\[
    \mathcal{A}_n = \quotient{\left(\prod_{p} \ZZ/p^n\ZZ \right)}{\left(\bigoplus_{p}\ZZ/p^n\ZZ \right)},
\]
where $p$ runs over all primes.
Then, the $\QQ$-algebra $\Ahat$ is defined as the projective limit of $\mathcal{A}_n$:
\[
    \Ahat = \lim_{\begin{subarray}{c} \longleftarrow\\ n \end{subarray}} \mathcal{A}_n.
\]
For each index $\bm{k} = (k_1, \dots, k_r)$, we define $\Ahat$-MZV $\zeta_{\Ahat}(\bm{k})$, as an element of $\Ahat$, by
\[
    \zeta_{\Ahat}(\bm{k}) \coloneq \left(\left( \sum_{0< m_1 < \cdots < m_r < p}\prod_{j = 1}^{r} \frac{1}{{m_j}^{k_j}} \ \bmod p^n \right)_p \right)_n
\]
and its star version $\zeta_{\Ahat}^{\star}(\bm{k})$ by
\[
    \zeta^{\star}_{\Ahat}(\bm{k}) \coloneq \left(\left( \sum_{0< m_1 \leq \cdots \leq m_r < p}\prod_{j = 1}^{r} \frac{1}{{m_j}^{k_j}} \ \bmod p^n \right)_p \right)_n = \sum_{\bm{l} \preceq \bm{k}}\zeta_{\Ahat}(\bm{l}).
\]
We denote by $\bm{p}$ an element of $\Ahat$ defined by
\[
    \bm{p} \coloneq \left((p \ \bmod p^n)_p\right)_n \in \Ahat.
\]
Furthermore, we define the finite multiple zeta value $\zeta_{\mathcal{A}}(\bm{k})$ to be the component corresponding to $n = 1$, as an element of $\mathcal{A} \coloneq \mathcal{A}_1$, namely, 
\[
    \zeta_{\mathcal{A}}(\bm{k}) = \left(\sum_{0< m_1 < \cdots < m_r < p}\prod_{j = 1}^{r} \frac{1}{{m_j}^{k_j}} \ \bmod p \right)_p \in \mathcal{A}.
\]

Next, we define $\Shat$-MZV $\zeta_{\Shat}(\bm{k})$ as an element of $\overline{\mathcal{Z}}[[t]]= \left(\quotient{\mathcal{Z}[\pi i]}{\pi i\mathcal{Z}[\pi i]}\right)[[t]]$.
For each index $\bm{k} = (k_1, \dots, k_r)$, we define
\[
    \zeta_{\Shat}(\bm{k}) \coloneq \sum_{a = 0}^{r}(-1)^{k_{a+1} + \cdots + k_r} \zeta^{\ast}(\bm{k}_{a} ; T)\sum_{l \geq 0}\sum_{\begin{subarray}{c} \bm{l} \in (\ZZ_{\geq 0})^{r-a} \\ \mathrm{wt}(\bm{l}) = l \end{subarray}} t^l ~b\binom{\bm{k}^{a}}{\bm{l}} \zeta^{\ast}(\overline{\bm{k}^{a} + \bm{l}}; T) \ \bmod \pi i\mathcal{Z}[\pi i], 
\]
where, $\zeta^{\ast}(\bm{k}; T)$ is the stuffle-regularized multiple zeta value, see \cite{IKZ}*{Section 2} for the definition.
Note that the sum in the definition does not depend on $T$.
The star version is defined similarly to $\Ahat$-MZV:
\[
    \zeta^{\star}_{\Shat}(\bm{k}) \coloneq \sum_{\bm{l} \preceq \bm{k}}\zeta_{\Shat}(\bm{l}).
\]
For details on the $t$-adic symmetric multiple zeta values, see \cite{OSY}*{Section 1}.
We define the symmetric multiple zeta value $\zeta_{\mathcal{S}}(\bm{k})$ as the $l = 0$ term of $\zeta_{\Shat}(\bm{k})$ in $\overline{\mathcal{Z}}$.

The refined Kaneko--Zagier conjecture is stated as follows.
\begin{conj}[\rm{cf.} \cite{Jar}*{Conjecture 5.3.2}, \cite{Ros2}*{Conjecture 2.3}]\label{conj:KZ}
    For a positive integer $k$ and rational numbers $\{c_{\bm{k}}^{(l)} \mid \mathrm{wt}(\bm{k}) = k+l,~l \geq 0\}$, we have
    \[
        \sum_{l\geq 0}\bm{p}^l\left( \sum_{\mathrm{wt}(\bm{k}) = k+l}c_{\bm{k}}^{(l)} \zeta_{\Ahat}(\bm{k})\right) = 0~ in~\Ahat \Longleftrightarrow  \sum_{l\geq 0}t^l\left( \sum_{\mathrm{wt}(\bm{k}) = k+l}c_{\bm{k}}^{(l)} \zeta_{\Shat}(\bm{k})\right) = 0~ in ~\overline{\mathcal{Z}}[[t]].
    \]
\end{conj}

\subsection{Notation}
For an index $\bm{k} = (k_1 ,\dots, k_r)$ and a tuple of non-negative integers $\bm{l} = (l_1,  \cdots , l_r)$, we define
\[
    \bm{k}_a = (k_1, \dots, k_a),~~ \bm{k}^{a} = (k_{a + 1}, \dots, k_r), ~~ \overline{\bm{k}} = (k_r, \dots, k_1),~~
\]
\[
    b\binom{\bm{k}}{\bm{l}} = \prod_{j = 1}^{r} \binom{k_j + l_j - 1}{l_j}
\]
and 
\[
    \bm{l}\oplus \bm{k} = (l_1 + k_1 ,\dots, l_r + k_r),~\bm{l} \oslash \bm{k} = (l_1 + 1, \{1\}^{k_1 - 1}, \dots l_r + 1, \{1\}^{k_r - 1}),
\]
where $\{k\}^{r}$ means $k, \dots , k$ repeated $r$ times.
Any non-empty index $\bm{k}$ can be written uniquely in the form
\[
\bm{k} = (\{1\}^{a_1 - 1}, b_1 + 1, \dots, \{1\}^{a_{s-1} - 1}, b_{s-1} + 1, \{1\}^{a_s- 1}, b_s)
\]
for positive integers $s$, $a_1,\dots , a_s$, and $b_1, \dots , b_s$.
Then, we define Hoffman's dual index $\bm{k}^{\lor}$ of $\bm{k}$ by 
\[
    \bm{k}^{\lor} = (a_1, \{1\}^{b_1 - 1}, a_2+1, \{1\}^{b_2 -1}, \dots , a_s+1, \{1\}^{b_s - 1}).
\]
For example, 
\[
    (2, 1, 3)^{\lor} = (\{1\}^{1-1},1 + 1, \{1\}^{2-1}, 3)^{\lor} = (1, \{1\}^{1-1}, 2 + 1, \{1\}^{3-1}) = (1, 3, 1, 1).
\]

Now, we consider the non-commutative polynomial ring $\mathfrak{H}^{1} = \QQ\langle z_k \mid k \in \ZZ_{>0} \rangle$.
We define the \emph{stuffle product} as a $\QQ$-bilinear operation $\ast : \mathfrak{H}^1 \times  \mathfrak{H}^1 \to \mathfrak{H}^1$ inductively by
\[
    wz_k \ast w'z_l = (w \ast w'z_l)z_k + (wz_k \ast w')z_l + (w \ast w')z_{k + l}
\]
and 
\[
    w \ast 1 = 1\ast w = w,
\]
where $w, w' \in \mathfrak{H}^1$ and $k, l > 0$.
As a $q$-analogue of this, we consider the $\QQ[1-q]$-algebra $\mathfrak{H}^1[1-q]$ and define the \emph{$q$-stuffle product} as a $\QQ[1-q]$-bilinear operation $\ast_q :\mathfrak{H}^1[1-q] \times \mathfrak{H}^1[1-q] \to \mathfrak{H}^1[1-q]$ by
\[
    wz_k \ast_q w'z_l = (w \ast_q w'z_l)z_k + (wz_k \ast_q w')z_l + (w \ast_q w')z_{k + l} + (1-q)(w\ast_qw')z_{k+l-1}
\]
and
\[
    w \ast_q 1 = 1\ast_q w = w.
\]

For an index $\bm{k} = (k_1, \dots, k_r)$, we set $z_{\bm{k}} = z_{k_1}\cdots z_{k_r}$, and we use $\bm{k}$ instead of $z_{\bm{k}}$ in $\mathfrak
{H}^{1}$.
We define the $\QQ[1-q]$-linear map $\zeta_{\Qhat}:\mathfrak{H}^1[1-q] \to \Qhat$ by $\zeta_{\Qhat}(z_{\bm{k}}) = \zeta_{\Qhat}(\bm{k})$, where we omit the symbols $BZ, SZ, \overline{\phantom{\zeta}}$ and $\star$.
See also \cref{def:Qhatmzv} for $\zeta_{\Qhat}$.
The $\QQ$-linear maps $\zeta_{\Ahat}$ and $\zeta_{\Shat}$ on $\mathfrak{H}^1$ are also defined similarly.

Furthermore, the $\QQ[1-q]$-linear map $\zeta^{BZ}_{\Qhat} : \mathfrak{H}^{1}[1-q] \to \Qhat$ is a homomorphism with respect to $\ast_q$.
For example, we have
\[
    \zeta_{\Qhat}^{BZ}(2)\zeta_{\Qhat}^{BZ}(1) = \zeta_{\Qhat}^{BZ}((2)\ast_q(1)) = \zeta_{\Qhat}^{BZ}(2, 1) + \zeta_{\Qhat}^{BZ}(1, 2) +\zeta_{\Qhat}^{BZ} (3) + (1-q)\zeta_{\Qhat}^{BZ}(2).
\]

In addition, we write $(\bm{k}, \bm{l})$ for the concatenation of the indices $\bm{k}, \bm{l}$, and extend it $\QQ[1-q]$-linearly.
For example, we write
\begin{align}
    ((2)\ast_q & (1, 1), 1) \\
    &= ((2, 1, 1)+ (1, 2, 1) + (1, 1, 2) + (3, 1) + (1, 3) + (1-q)(1, 2) + (1-q)(2, 1), 1)\\
    &= (2, 1, 1, 1)+ (1, 2, 1, 1) + (1, 1, 2, 1) + (3, 1, 1) + (1, 3, 1) \\
    &\hspace{9.0cm}+ (1-q)\{(1, 2, 1) + (2, 1, 1)\}.
\end{align}

Finally, we write $\bm{l} \preceq \bm{k}$ if $\bm{l}$ obtained by replacing some commas in $\bm{k} = (k_1,\dots, k_r)$ by plus signs.

\section{The algebra $\Qhat$ and $\Qhat$-multiple zeta values}
In this section, we consider the $\QQ[q]$-algebra $\Qhat$ and $\Qhat$-multiple zeta values.
Now, we recall the definition of $\Qhat$ and the maps $\phi_{\Ahat}$ and $\phi_{\Shat}$ introduced in \cite{TT}.

For a positive integer $m$, we denote the $q$-integer $[m]$ by
\[
    [m] \coloneq [m]_q \coloneq \frac{1-q^m}{1-q}.
\]

\subsection{Definition of the algebra $\Qhat$}
We denote by $\ZZ_{(p)}[q]$ the polynomial ring in $q$ over the ring $\ZZ_{(p)}$, where $\ZZ_{(p)}$ is the localization of $\ZZ$ at a prime $p$.
For $n \geq 1$, we consider the rings
\[
    Z_{p, n} \coloneq \quotient{\ZZ_{(p)}[q]}{([p]^n)} \quad \text{ and } \quad\, \mathcal{Q}_n \coloneq \quotient{\left(\prod_{p} Z_{p, n}\right)}{\left(\bigoplus_{p} Z_{p,n}\right)}
\]
where $([p]^n)$ is the ideal of $\ZZ_{(p)}[q]$ generated by $[p]^n$ and $p$ runs over all primes.
Note that, equipped with diagonal embedding of $\QQ[q]$ into $\mathcal{Q}_n$ and component-wise addition, multiplication, and scalar multiplication by $\QQ[q]$, $\mathcal{Q}_n$ becomes a $\QQ[q]$-algebra.
The projection $\ZZ_{p, n+1} \to \ZZ_{p, n}$ naturally induces the transition map $\varphi_{n}:\mathcal{Q}_{n+1} \to \mathcal{Q}_n$.
Then we define the $\QQ[q]$-algebra $\Qhat$ to be the projective limit of the system $\{\mathcal{Q}_n, \varphi_n \}$:
\[
    \Qhat \coloneq \lim_{\begin{subarray}{c} \longleftarrow\\ n \end{subarray}} \mathcal{Q}_n.
\]
We write any element of $\Qhat$ as $\big((f_{p, n})_p \big)_n$ with $(f_{p, n})_p \in \mathcal{Q}_n$, $f_{p,n} \in Z_{p, n}$.
We define the element $[\bm{p}]$ of $\Qhat$ by
\[
    [\bm{p}] = \big(([p] \ \bmod [p]^n)_p\big)_n.
\]

\begin{rem}
\rm{We remark the following:
    \begin{itemize}
        \item[(1)] For two elements $(f_{p, n})_p, (g_{p, n})_p$ of $\mathcal{Q}_n$  , $(f_{p, n})_p = (g_{p, n})_p$ holds if and only if $f_{p, n} = g_{p, n}$ holds for all but finitely many primes p.
        \item[(2)] For two elements $\big((f_{p, n})_p\big)_n, \big((g_{p, n})_p\big)_n$ of $\Qhat$, $\big((f_{p, n})_p\big)_n = \big((g_{p, n})_p\big)_n$ holds if and only if there exists an increasing sequence $\{p_n\}_n$ of primes such that $f_{p, n} = g_{p, n}$ in $Z_{p, n}$ for all $p \geq p_n$ and $n \geq 1$.
        \item [(3)] The projective topology on $\Qhat$ induced from the discrete topology on $\mathcal{Q}_n$ coincides with the $[\bm{p}]$-adic topology, and $\Qhat$ is complete under the $[\bm{p}]$-adic topology.
    \end{itemize}}
\end{rem}

\subsection{Construction of the map $\phi_{\Ahat}$}
In the following way, $\Qhat$ can be regarded as the $q$-analogue of $\Ahat$.
If $f(q) \equiv g(q) \ \bmod [p]^n$ for $f(q), g(q) \in \ZZ_{(p)}[q]$, then we have $f(1) \equiv g(1) \ \bmod p^n$.
Therefore, there is the natural projection 
\[
    Z_{p, n} \longrightarrow ~ \quotient{\ZZ_{(p)}}{p^n\ZZ_{(p)}} ~ \cong ~ \quotient{\ZZ}{p^n\ZZ}
\]
that sends $q \mapsto 1$.
This induces algebra homomorphism $\phi_{\Ahat}$
\begin{equation}
    \begin{array}{r@{\,\,}c@{\,\,}c@{\,\,}c}
        \phi_{\Ahat}~\colon & \Qhat &\longrightarrow   & \Ahat \\
                     &\rotatebox{90}{$\in$}&&\rotatebox{90}{$\in$}\\
                     & \big((f_{p, n}(q) \ \bmod [p]^n)_p\big)_n &\longmapsto & \big((f_{p, n}(1) \ \bmod p^n)_p\big)_n.
    \end{array}
\end{equation}

We call it the algebraic limit.

\subsection{Construction of the map $\phi_{\Shat}$}
To define the map $\phi_{\Shat}$, we first define two maps $\widehat{ev}$ and $\widehat{lim}$.

We define $\widehat{\mathcal{Q}}^{\mathrm{an}}$ as in the definition of $\Qhat$.
For a prime $p$ and $n \geq 1$, 
\[
     Z^{\mathrm{an}}_{p, n} \coloneq \quotient{\QQ(\zeta_p)[[t]]}{(t^n)} \text{ and } \, \mathcal{Q}^{\mathrm{an}}_n \coloneq \quotient{\left(\prod_{p} Z^{\mathrm{an}}_{p, n}\right)}{\left(\bigoplus_{p} Z^{\mathrm{an}}_{p,n}\right)}.
\]
The transition map $\varphi^{\mathrm{an}}_{n} : \mathcal{Q}^{\mathrm{an}}_{n+1} \to \mathcal{Q}^{\mathrm{an}}_n$ is induced naturally by projection $Z^{\mathrm{an}}_{p, n+1} \to Z^{\mathrm{an}}_{p, n}$, and we set
\[
    \widehat{\mathcal{Q}}^{\mathrm{an}} \coloneq \lim_{\begin{subarray}{c} \longleftarrow\\ n \end{subarray}} \mathcal{Q}^{\mathrm{an}}_n.
\]
For each integer $m \geq 2$, there exists a unique formal power series $q_m(t) \in \QQ(\zeta_m)[[t]]$ such that
\[
    q_m(0) = \zeta_m , \quad [m]_{q_{m}(t)} = t.
\]
In \cite{TT}*{Proposition A.2.}, the asymptotic formula of $q_{m}(t)$ is obtained.
In particular, the coefficient of $t^l$ for $l \geq 1$ is of the order $O(m^{-(l+1)})$ as $m \to \infty$.
Using $q_p(t)$ for each prime $p$, we define the $\ZZ_{(p)}$-homomorphism
\begin{equation}
    \begin{array}{c@{\,\,}c@{\,\,}c}
         Z_{p, n} &\longrightarrow   & Z^{\mathrm{an}}_{p, n} \\
         \rotatebox{90}{$\in$}&&\rotatebox{90}{$\in$}\\
         f(q) \ \bmod [p]^n &\longmapsto & f(q_p(t)) \ \bmod t^n.
    \end{array}
\end{equation}
This induces the $\QQ$-algebra map
\begin{equation}
    \begin{array}{r@{\,\,}c@{\,\,}c@{\,\,}c}
        \widehat{ev}~\colon & \Qhat &\longrightarrow   & \widehat{\mathcal{Q}}^{\mathrm{an}}  \\
         &\rotatebox{90}{$\in$}&&\rotatebox{90}{$\in$}\\
         & \big((f_{p, n}(q) \ \bmod [p]^n)_p\big)_n &\longmapsto & \big((f_{p, n}(q_p(t)) \ \bmod t^n)_p\big)_n.
    \end{array}
\end{equation}

Next, we define the $\widehat{lim}$.
We denote $\mathcal{O}^{\mathrm{an}}_n$ by the $\QQ$-subalgebra of $\textstyle\prod_{p}Z^{\mathrm{an}}_{p, n}$ such that
\[
    \mathcal{O}^{\mathrm{an}}_n = \left\{\left(\sum_{l = 0}^{n-1} z_{p, l}t^l \ \bmod t^n \right)_p \in \prod_{p}Z^{\mathrm{an}}_{p, n} ~\middle| ~ \lim_{p \to \infty} z_{p, l} ~\text{converges for all}~0 \leq l < n \right\}.
\]
We see that $\textstyle\bigoplus_{p}Z^{\mathrm{an}}_{p, n} \subset \mathcal{O}^{\mathrm{an}}_n\subset \textstyle\prod_{p}Z^{\mathrm{an}}_{p, n}$, thus we define the $\QQ$-subalgebra $\widehat{\mathcal{O}}^{\mathrm{an}}$ of $\widehat{\mathcal{Q}}^{\mathrm{an}}$ by
\[
    \widehat{\mathcal{O}}^{\mathrm{an}}  \coloneq \lim_{\begin{subarray}{c} \longleftarrow\\ n \end{subarray}} \left(\quotient{\mathcal{O}^{\mathrm{an}}_n}{\bigoplus_{p}Z^{\mathrm{an}}_{p, n}}\right), 
\]
where the transition map 
\[
    \psi^{\mathrm{an}}_{n} : \quotient{\mathcal{O}^{\mathrm{an}}_{n+1}}{\bigoplus_{p}Z^{\mathrm{an}}_{p, n+1}} \to \quotient{\mathcal{O}^{\mathrm{an}}_n}{\bigoplus_{p}Z^{\mathrm{an}}_{p, n}}
\]
is also induced by the projection $Z^{\mathrm{an}}_{p, n+1} \to Z^{\mathrm{an}}_{p, n}$.
Any element $z$ of $\widehat{\mathcal{O}}^{\mathrm{an}}$ is of the form
\[
    z = (z_n)_n = \left( \left(\sum_{l = 0}^{n-1} z_{p, l}^{(n)} t^l \ \bmod t^n \right)_p \ \bmod \bigoplus_{p}Z^{\mathrm{an}}_{p, n}\right)_n, 
\]
and we have  $\psi^{\mathrm{an}}_{n}(z_{n + 1}) = z_n$, that is, 
\[
    \left( \sum_{l = 0}^{n} z_{p, l}^{(n+1)} t^l \ \bmod t^n \right)_p \equiv \left(\sum_{l = 0}^{n-1} z_{p, l}^{(n)} t^l \ \bmod t^n \right)_p \ \bmod \bigoplus_{p}Z^{\mathrm{an}}_{p, n}.
\]
Therefore, there exists an increasing sequence $\{p_n\}_n$ such that equality $z_{p, l}^{(n+1)} = z_{p, l}^{(n)}$ holds for primes $p \geq p_n$ and each $l \in \{0, 1, \dots , n-1\}$.
Taking $p \to \infty$, we have
\[
    \lim_{p\to \infty}z_{p, l}^{(n+1)} = \lim_{p \to \infty}z_{p, l}^{(n)}.
\]
Thus, we define
\[
    z_l \coloneq \lim_{p \to \infty} z_{p, l}^{(n)}
\]
for $l \geq 0$, which is independent of the choice of $n$ greater than $l$.
Hence, we obtain the well-defined map
\begin{equation}
    \begin{array}{r@{\,\,}c@{\,\,}c@{\,\,}c}
            \widehat{lim}~\colon & \widehat{\mathcal{O}}^{\mathrm{an}}&\longrightarrow & \CC[[t]] \\
             &\rotatebox{90}{$\in$}&&\rotatebox{90}{$\in$}\\
             & z &\longmapsto & \sum_{l\geq 0}z_l t^l.
    \end{array}
\end{equation}

Finally, define the $\QQ$-subalgebra $\widehat{\mathcal{O}}$ of $\Qhat$ by
\[
    \widehat{\mathcal{O}} = \widehat{ev}^{-1}(\widehat{\mathcal{O}}^{\mathrm{an}}).
\]
Then, the map $\phi_{\Shat}$ is defined as the composition of $\widehat{lim}$ and $\widehat{ev}$, explicitly, 
\begin{equation}
    \begin{array}{r@{\,\,}c@{\,\,}c@{\,\,}c}
            \phi_{\Shat}~\colon & \widehat{\mathcal{O}}&\longrightarrow & \CC[[t]] \\
             &\rotatebox{90}{$\in$}&&\rotatebox{90}{$\in$}\\
             & \left((f_{p, n}(q) \ \bmod [p]^n)_p\right)_n &\longmapsto & \sum_{l\geq 0}z_l t^l, 
    \end{array}
\end{equation}
where $f_{p, n}(q_p(t)) = \sum_{l=0}^{n-1} z_{p, l}^{(n)}t^l + O(t^n)$ and $z_l = \lim_{p \to \infty}z_{p, l}^{(n)}~(n>l)$.
We call it the analytic limit.

\subsection{Various models of  $\Qhat$-MZV}
In this subsection, we define various models of $\Qhat$-multiple zeta values ($\Qhat$-MZVs for short), and we see what values they take under the maps $\phi_{\Ahat}$ and $\phi_{\Shat}$.

\begin{dfn}\label{def:Qhatmzv}
    \rm{For an index $\bm{k} = (k_1, \dots , k_r)$ and $\bullet \in \{\emptyset, \star\}$, we define
    \begin{align}
        \zeta_{\Qhat}^{BZ, \bullet}(\bm{k}) &\coloneq \left(\left( \zeta_{p-1}^{BZ, \bullet}(\bm{k}) \ \bmod [p]^n \right)_p\right)_n, \\
        \zeta_{\Qhat}^{SZ, \bullet}(\bm{k}) &\coloneq \left(\left( \zeta_{p-1}^{SZ, \bullet}(\bm{k}) \ \bmod [p]^n \right)_p\right)_n, \\
        \overline{\zeta}^{\bullet}_{\Qhat}(\bm{k}) &\coloneq \left(\left( \overline{\zeta}_{p-1}^{\bullet}(\bm{k}) \ \bmod [p]^n \right)_p\right)_n
    \end{align}
    as the elements of $\Qhat$, where 
    \begin{align}
        \zeta_{N}^{BZ}(\bm{k}) &\coloneq \sum_{0 < m_1 < \cdots < m_r \leq N}\prod_{j = 1}^{r} \frac{q^{(k_j-1)m_j}}{[m_j]^{k_j}}, &\zeta_{N}^{BZ, \star}(\bm{k}) &\coloneq \sum_{0 < m_1 \leq \cdots\leq m_r\leq N}\prod_{j = 1}^{r} \frac{q^{(k_j-1)m_j}}{[m_j]^{k_j}}, \\ 
        \zeta_{N}^{SZ}(\bm{k}) &\coloneq \sum_{0 < m_1 < \cdots < m_r \leq N}\prod_{j = 1}^{r} \frac{q^{k_j m_j}}{[m_j]^{k_j}}, &\zeta_{N}^{SZ, \star}(\bm{k}) &\coloneq \sum_{0 < m_1 \leq \cdots \leq m_r \leq N}\prod_{j = 1}^{r} \frac{q^{k_j m_j}}{[m_j]^{k_j}}, \\
        \overline{\zeta}_{N}(\bm{k}) &\coloneq \sum_{0 < m_1 < \cdots < m_r \leq N}\prod_{j = 1}^{r} \frac{q^{m_j}}{[m_j]^{k_j}}, &\overline{\zeta}_{N}^{\star}(\bm{k}) &\coloneq \sum_{0 < m_1 \leq \cdots \leq m_r \leq N}\prod_{j = 1}^{r} \frac{q^{m_j}}{[m_j]^{k_j}}.
    \end{align}}
\end{dfn}
\begin{rem}
\rm{
    For any prime $p$, integers $0 < m < p$ and $n \geq 1$,  we see that
    \[
        \frac{1}{[m]} \equiv [s]_{q^m}\sum_{j = 0}^{n-1} (-q[t]_{q^p}[p])^j \, \ \bmod [p]^n
    \]
    where $ms - pt = 1$ with $0 < s < p, t \geq 0$.
    Hence, since $[m]$ is invertible in $Z_{p, n}$, the above definition makes sense.
    }
\end{rem}

Note that the name "$BZ$" comes from Bradley and Zhao, who studied the non-truncated version of $\zeta_{N}^{BZ}(\bm{k})$ as a $q$-analogue of multiple zeta values.
Similarly, the name "$SZ$" is named after Schlesinger and Zudilin.
The models of $q$-MZVs are discussed in detail in \cite{Bri}.

These various multiple harmonic $q$-sums differ only in the numerators $q^{s_jm_j}\, (s_j \in \ZZ)$, but each of them plays an important role in the duality formulas for $\Qhat$-MZVs as stated in the introduction.
However, by the maps $\phi_{\Ahat}$ and $\phi_{\Shat}$, we see that these differences disappear in a certain sense.

First, for $\phi_{\Ahat}$, it is obvious that the following holds.

\begin{thm}[\cite{TT}*{Theorem 4.3}]
    For any index $\bm{k}$ and $\bullet \in \{\emptyset, \star\}$, we have
    \[
        \phi_{\Ahat}(\zeta_{\Qhat}^{BZ, \bullet}(\bm{k})) = \phi_{\Ahat}(\zeta_{\Qhat}^{SZ, \bullet}(\bm{k})) = \phi_{\Ahat}(\overline{\zeta}_{\Qhat}^{\bullet}(\bm{k})) = \zeta_{\Ahat}^{\bullet}(\bm{k}).
    \]
\end{thm}

Next,  we consider $\phi_{\Shat}$.
\begin{thm}[\cite{TT}*{Theorem 4.4}]
    Define
    \[
        \widehat{\xi}(\bm{k}) \coloneq \sum_{a = 0}^{r}(-1)^{k_{a+1} + \cdots + k_r} \zeta^{\ast}\left(\bm{k}_{a} ; -\frac{\pi i}{2}\right)\sum_{l \geq 0}\sum_{\begin{subarray}{c} \bm{l} \in (\ZZ_{\geq0})^{r-a} \\ \mathrm{wt}(\bm{l}) = l \end{subarray}} t^l ~b\binom{\bm{k}^{a}}{\bm{l}} \zeta^{\ast}\left(\overline{\bm{k}^{a} + \bm{l}}; \frac{\pi i}{2}\right) \in \mathcal{Z}[\pi i][[t]]
    \]
    and
    \[
        \widehat{\xi}^{\star}(\bm{k}) \coloneq \sum_{\bm{l} \preceq \bm{k}} \widehat{\xi}(\bm{l}).
    \]
    Then, for any index $\bm{k}$ and $\bullet \in \{\emptyset, \star\}$, we have
    \[
        \phi_{\Shat}(\zeta_{\Qhat}^{BZ, \bullet}(\bm{k})) = \widehat{\xi}^{\bullet}(\bm{k}) \quad \text{and} \quad \phi_{\Shat}(\overline{\zeta}_{\Qhat}^{\bullet}(\bm{k})) = \overline{\widehat{\xi}^{\bullet}(\bm{k})},
    \]
    where the bar on the right-hand side means taking the complex conjugate of each coefficient over $\CC$.
    In particular, it holds that
    \[
        \phi_{\Shat}(\zeta_{\Qhat}^{BZ, \bullet}(\bm{k})) \equiv  \phi_{\Shat}(\overline{\zeta}_{\Qhat}^{\bullet}(\bm{k})) \equiv \zeta_{\Shat}^{\bullet}(\bm{k}) 
        \ \bmod \pi i \mathcal{Z}[\pi i][[t]].
    \]
\end{thm}

The models of $\Qhat$-MZVs considered in \cite{TT} were limited to $\zeta_{\Qhat}^{BZ}$ and $\overline{\zeta}_{\Qhat}$, which appeared in \cref{prop:tt}. 
We show that this evaluation can also be applied to $\zeta_{\Qhat}^{SZ}$.

\begin{thm}
    For any index $\bm{k}$ and $\bullet \in \{\emptyset, \star\}$, we have
    \[
        \phi_{\Shat}(\zeta_{\Qhat}^{SZ, \bullet}(\bm{k})) = \overline{\widehat{\xi}^{\bullet}(\bm{k})},
    \]
    and
    \[
        \phi_{\Shat}(\zeta_{\Qhat}^{SZ, \bullet}(\bm{k})) \equiv \zeta_{\Shat}^{\bullet}(\bm{k}) 
        \ \bmod \pi i \mathcal{Z}[\pi i][[t]].
    \]
\end{thm}
\begin{proof}
By using the identity
\[
    q^{km} = q^{m}\{1-(1-q)[m]\}^{k-1} = q^m + q^m\sum_{j=1}^{k-1}\binom{k-1}{j}\{-(1-q)[m]\}^{j},
\]
$\zeta_{\Qhat}^{SZ}(\bm{k})$ can be expressed as a $\QQ[1-q]$-linear combination of $\overline{\zeta}_{\Qhat}(\bm{l})$.
More precisely, if an index $\bm{k}$ has an entry greater than $1$, we have
    \begin{align}
         \zeta_{\Qhat}^{SZ}(\bm{k}) - \overline{\zeta}_{\Qhat}(\bm{k}) \in \Span_{\QQ}\{(1-q)^{j}\overline{\zeta}_{\Qhat}(\bm{l}) \mid 0 < j < \mathrm{wt}(\bm{k}), \mathrm{wt}(\bm{l}) + j = \mathrm{wt}(\bm{k})\}.
    \end{align}
For example, we have
    \begin{align}
        \zeta_{\Qhat}^{SZ}(2, 1, 3) = \overline{\zeta}_{\Qhat}(2, 1, 3) &-(1-q)\{2\overline{\zeta}_{\Qhat}(2, 1, 2) + \overline{\zeta}_{\Qhat}(1, 1, 3)\} \\
        &+ (1-q)^2\{\overline{\zeta}_{\Qhat}(2, 1, 1) + 2\overline{\zeta}_{\Qhat}(1, 1, 2)\} - (1-q)^3\overline{\zeta}_{\Qhat}(1, 1, 1).
    \end{align}
Note that if $\bm{k} = (\{1\}^r)$ for some $r$, then $\zeta_{\Qhat}^{SZ}(\bm{k}) = \overline{\zeta}_{\Qhat}(\bm{k})$ by definition.

    Since $\phi_{\Shat}(1-q) = 0$, we obtain
    \[
    \phi_{\Shat}(\zeta_{\Qhat}^{SZ, \bullet}(\bm{k})) = \phi_{\Shat}(\overline{\zeta}_{\Qhat}^{\bullet}(\bm{k})) \equiv \zeta_{\Shat}^{\bullet}(\bm{k}) 
        \ \bmod \pi i \mathcal{Z}[\pi i][[t]].
    \]
    by the above discussion.
\end{proof}

From the above theorems, once we obtain a $\QQ[q]$-linear relation for $\Qhat$-MZVs, $\QQ$-linear relations for $\Ahat$-MZV and $\Shat$-MZV follow simultaneously via the maps $\phi_{\Ahat}$ and $\phi_{\Shat}$, respectively.

\section{duality formulas}
In this section, we introduce two duality formulas for $\Qhat$-MZVs.
\subsection{The $q$-analogue of Rosen's duality formula}
The following is the $q$-analogue of \cref{prop:rosen}.
\begin{thm}\label{thm:qrosen}
     For an index $\bm{k}$, it holds that
    \begin{align}
        q^{\frac{\bm{p}(\bm{p} - 1)}{2}}\zeta_{\Qhat}^{BZ}(\bm{k}) + \sum_{l \geq 0}(-1)^{l} \zeta_{\Qhat}^{BZ}(\bm{k} \ast_q \{1\}^l, 1)[\bm{p}]^{l+1} = (-1)^{\mathrm{dep}(\bm{k})} \sum_{\bm{k} \preceq \bm{l}} \zeta_{\Qhat}^{SZ}(\bm{l}),
    \end{align}
where
    \[q^{\frac{\bm{p}(\bm{p} - 1)}{2}} := \left(\left( \sum_{l = 0}^{n - 1}\binom{\frac{p-1}{2}}{l}(-(1-q)[p])^l \ \bmod [p]^n\right)_p\right)_n.\]
\end{thm}

Our proof of \cref{thm:qrosen} is based on a $q$-analogue of the proof for \cref{prop:rosen} in \cite{Ros1}*{Theorem 4.5}.
We need some auxiliary results.

\begin{lem}\label{prop:sumqbinom}
    For integers $m, N$ with $0 < m \leq N$, we have
    \[\sum_{i = m+1}^{N + 1} (-1)^{i + 1}q^{\binom{i}{2}}\binom{N+1}{i}_q = (-1)^mq^{\binom{m + 1}{2}}\binom{N}{m}_q.\]
\end{lem}
\begin{proof}
    Using the $q$-binomial theorem, it holds that
     \[0 = \sum_{i = 0}^{N + 1}(-1)^{i}q^{\binom{i}{2}}\binom{N+1}{i}_q.\]
    Therefore,
    \begin{align}
       \sum_{i = m+1}^{N + 1} (-1)^{i + 1}q^{\binom{i}{2}}\binom{N+1}{i}_q &= \sum_{i = 0}^{m}(-1)^{i}q^{\binom{i}{2}}\binom{N+1}{i}_q \\
       &= \sum_{i = 0}^{m} \left\{(-1)^{i}q^{\binom{i+1}{2}}\binom{N}{i}_q - (-1)^{i-1}q^{\binom{i}{2}}\binom{N}{i-1}_q \right\} \\
       &= (-1)^{m}q^{\binom{m+1}{2}}\binom{N}{m}_q, 
    \end{align}
    where for the second equality, we have used the identity
    \[\binom{N + 1}{i}_q = q^{i}\binom{N}{i}_q + \binom{N}{i-1}_q.\]
    This completes the proof.
\end{proof}

\begin{prop}\label{prop:qhoffman}
    For an integer $N > 0$ and an index $\bm{k}$, we have
    \[
        H_N(\bm{k}) = \sum_{i = 1}^{N+1}(-1)^{i+1}q^{\binom{i}{2}}\binom{N+1}{i}_q\zeta_{i-1}^{BZ}(\bm{k}), \]
    where 
    \[
        H_N({\bm{k}}) \coloneq \sum_{0 < m_1 < \cdots < m_r \leq N}\frac{q^{(k_1-1)m_1 + \cdots + (k_r-1)m_r}}{[m_1]^{k_1}\cdots [m_r]^{k_r}} (-1)^{m_r}q^{\binom{m_r + 1}{2}} \binom{N}{m_r}_q.
    \]
\end{prop}
\begin{proof}
    Calculating the right-hand side by changing the order of summation and using \\ 
    \cref{prop:sumqbinom}, we obtain
    \begin{align}
      \sum_{i = 1}^{N+1}(-1)^{i+1}q^{\binom{i}{2}} & \binom{N+1}{i}_q \zeta_{i-1}^{BZ}(\bm{k}) \\
      &= \sum_{0 < m_1 < \cdots < m_r < i \leq N + 1} (-1)^{i+1}q^{\binom{i}{2}}\binom{N+1}{i}_q \prod_{j = 1}^{r} \frac{q^{(k_j-1)m_j}}{[m_j]^{k_j}} \\
      &= \sum_{0 < m_1 < \cdots < m_r \leq N} \prod_{j = 1}^{r} \frac{q^{(k_j-1)m_j}}{[m_j]^{k_j}} \sum_{i = m_r + 1}^{N+1}(-1)^{i+1}q^{\binom{i}{2}}\binom{N+1}{i}_q \\
      &= \sum_{0 < m_1 < \cdots < m_r \leq N}\prod_{j = 1}^{r} \frac{q^{(k_j-1)m_j}}{[m_j]^{k_j}}(-1)^{m_r}q^{\binom{m_r + 1}{2}} \binom{N}{m_r}_q \\
      &= H_N(\bm{k}).
    \end{align}
    This is the desired result.
\end{proof}
\begin{prop}\label{prop:qrosen}
    For any prime $p$ and index $\bm{k}$, we have
    \[\sum_{i=1}^{p} (-1)^{i+1}q^{\binom{i}{2}}\binom{p}{i}_q\zeta_{i-1}^{BZ}(\bm{k}) = (-1)^{p+1}q^{\binom{p}{2}}\zeta_{p-1}^{BZ}(\bm{k}) + \sum_{l \geq 0} (-1)^{l}\zeta_{p-1}^{BZ}(\bm{k} \ast_q \{1\}^l, 1)[p]^{l+1}.\]
\end{prop}
\begin{proof}
    We see that
    \[[p - j] = q^{-j}([p] - [j])\]
    for integers $j, p$ with $0 < j < p$.
    Then, we obtain
    \begin{align}
      q^{\binom{i}{2}}\binom{p}{i}_q &= q^{\binom{i}{2}}\frac{[p]}{[i]} \prod_{j = 1}^{i-1} \frac{[p-j]}{[j]}\\
      &= (-1)^{i-1}\frac{[p]}{[i]}\prod_{j=1}^{i-1}\left(1 - \frac{[p]}{[j]} \right) \\
      &= (-1)^{i-1}\frac{1}{[i]}\sum_{l \geq 0}(-1)^{l} \zeta_{i-1}^{BZ}(\{1\}^{l})[p]^{l+1}.
    \end{align}
     This implies that
     \begin{align}
      \sum_{i=1}^{p} (-1)^{i+1}q^{\binom{i}{2}}\binom{p}{i}_q\zeta_{i-1}^{BZ}(\bm{k}) &= (-1)^{p+1}q^{\binom{p}{2}}\zeta_{p-1}^{BZ}(\bm{k}) + \sum_{i= 1}^{p-1} \frac{1}{[i]}\sum_{l \geq 0}(-1)^{l}\zeta_{i-1}^{BZ}(\bm{k})\zeta_{i-1}^{BZ}(\{1\}^{l})[p]^{l+ 1} \\
      &= (-1)^{p+1}q^{\binom{p}{2}}\zeta_{p-1}^{BZ}(\bm{k}) + \sum_{l \geq 0} (-1)^{l}\zeta_{p-1}^{BZ}(\bm{k} \ast_q \{1\}^l, 1)[p]^{l+1}
    \end{align}
    and completes the proof.
\end{proof}
\begin{thm}[\cite{HHT}*{Theorem 8.1.}]\label{prop:HHT}
    For an integer $N > 0$ and an index $\bm{k}$, we have
    \[H_N(\bm{k}) = (-1)^{\mathrm{dep}(\bm{k})} \sum_{\bm{k} \preceq \bm{l}} \zeta_{N}^{SZ}(\bm{l}).\]
\end{thm}
\begin{proof}[Proof of \cref{thm:qrosen}]
    Now substitute $N = p - 1$ in \cref{prop:qhoffman} and \cref{prop:HHT}.
    Combining the above formulas and taking modulo $[p]^n$ for $n \geq 1$, we obtain the desired result.
\end{proof}

Applying the maps $\phi_{\Ahat}$ and $\phi_{\Shat}$ to \cref{thm:qrosen}, we have the following result.
\begin{cor}[=\cref{prop:rosen} (\cite{Ros1}*{Theorem 4.5})]
    For an index $\bm{k} = (k_1, \cdots, k_r)$, we have
        \[
            \zeta_{\Ahat}(\bm{k}) + \sum_{l \geq 0} (-1)^{l}\zeta_{\Ahat}(\bm{k}\ast \{1\}^l, 1)\bm{p}^{l + 1} = (-1)^{\mathrm{dep}(\bm{k})} \sum_{\bm{k} \preceq \bm{l}} \zeta_{\Ahat}(\bm{l}).
        \]
\end{cor}

\begin{cor}
    For an index $\bm{k} = (k_1, \cdots, k_r)$, we have
    \[
        e^{\pi i t}\widehat{\xi}(\bm{k}) + \sum_{l \geq 0} (-1)^l\widehat{\xi}(\bm{k}\ast \{1\}^l, 1)t^{l + 1} = (-1)^{\mathrm{dep}(\bm{k})} \sum_{\bm{k} \preceq \bm{l}} \overline{\widehat{\xi}(\bm{l})}
    \]
    and
    \[
        \zeta_{\Shat}(\bm{k}) + \sum_{l \geq 0} (-1)^{l}\zeta_{\Shat}(\bm{k}\ast \{1\}^l, 1)t^{l + 1} = (-1)^{\mathrm{dep}(\bm{k})} \sum_{\bm{k} \preceq \bm{l}} \zeta_{\Shat}(\bm{l}).
    \]
\end{cor}
\begin{proof}
   We see that 
    \[
        \phi_{\Shat}(q^{\frac{\bm{p}(\bm{p} - 1)}{2}}) = e^{\pi i t}
    \]
from \cite{TT}*{Appendix B.2.}.
    Applying this to \cref{thm:qrosen}, we obtain the desired result.
\end{proof}

The result for $\Shat$-MZV can be viewed as an explicit formulation of \cite{Jar}*{Proposition 4.3.1}.
Our proof only uses series manipulations and the limit argument.

\subsection{The $q$-analogue of Maesaka--Seki--Watanabe's duality formula}
The following is the $q$-analogue of \cref{prop:msw}.
\begin{thm}\label{thm:qmsw}
    For an index $\bm{k}$, we have
    \[
    \zeta_{\Qhat}^{BZ}(\bm{k}) = (-q^{-\bm{p}})^{\mathrm{dep}( \bm{k})} \sum_{l \geq 0} \left(\sum_{\begin{subarray}{c} \bm{l} \in (\mathbb{Z}_{\geq 0})^{\mathrm{dep}( \bm{k})} \\ \mathrm{wt}(\bm{l}) = l \end{subarray}}\sum_{\bm{l} \oplus \bm{k} \preceq \bm{m} \preceq \bm{l} \oslash \bm{k}} \zeta_{\Qhat}^{SZ}(\bm{m}) \right)([\bm{p}]q^{-\bm{p}})^{l},
    \]
    where $q^{-\bm{p}} = \sum_{j \geq 0}([\bm{p}](1-q))^j  \in \Qhat$.
\end{thm}

Our proof for \cref{thm:qmsw} is also based on the $q$-analogue of the proof for \cref{prop:msw} in \cite{MSW}*{Theorem 5.3}.
Therefore, the following formula, which is a $q$-analogue of the MSW formula, plays an important role.

\begin{thm}[\cite{Tsu}*{Theorem 1.2}]\label{prop:qmsw}
    For an integer $N>0$ and an index $\bm{k} = (k_1, \dots, k_r)$, it holds that
    \begin{align}
        \zeta_{N}^{BZ}(\bm{k}) = \sum_{\begin{subarray}{c} 0 < n_{j, 1} \leq \cdots \leq n_{j, k_j} \leq N \\ n_{j, k_j} < n_{j+1, 1} \end{subarray}} \prod_{j = 1}^{r} \frac{q^{n_{j, 2} + \cdots + n_{j, k_j}}}{[N + 1 - n_{j, 1}][n_{j, 2}]\cdots [n_{j, k_j}]}.
    \end{align}
\end{thm}

\begin{proof}[Proof of \cref{thm:qmsw}]
    For integers $n, p$ with $0 < n< p$, we see that
    \begin{align}
      \frac{1}{[p-n]} &= -\frac{1}{q^{p-n}[n]}\frac{1}{1- \frac{[p]}{q^{p-n}[n]}} \\
      &= -\frac{1}{q^{p}}\frac{q^n}{[n]}\sum_{l = 0}^{\infty} \frac{q^{nl}}{q^{pl}}\frac{[p]^{l}}{[n]^{l}} \\
      &= -\sum_{l \geq 0} \frac{q^{(l+1)n}}{[n]^{l+1}} q^{-p(l+1)}[p]^{l}.
    \end{align}
    Now consider the case $N = p - 1$ in \cref{prop:qmsw}. 
    Calculate the summand by using the above expansion and simplify the sum, we obtain the desired result.
\end{proof}

Applying the maps $\phi_{\Ahat}$ and $\phi_{\Shat}$ to \cref{thm:qmsw}, we have the following result.
\begin{cor}[=\cref{prop:msw} (\cite{MSW}*{Theorem 5.3})]
    For an index $\bm{k} = (k_1, \cdots , k_r)$, we have
        \[
            \zeta_{\Ahat}(\bm{k}) = (-1)^{\mathrm{dep}(\bm{k})} \sum_{l \geq 0} \left(\sum_{\begin{subarray}{c} \bm{l} \in (\mathbb{Z}_{\geq 0})^{\mathrm{dep}(\bm{k})} \\ \mathrm{wt}(\bm{l}) = l \end{subarray}}\sum_{\bm{l} \oplus \bm{k} \preceq \bm{m} \preceq \bm{l} \oslash \bm{k}} \zeta_{\Ahat}(\bm{m}) \right)\bm{p}^{l}.
        \]
\end{cor}
\begin{cor}[=\cref{cor:Smsw}]
     For an index $\bm{k} = (k_1, \cdots , k_r)$, we have
     \[
         \widehat{\xi}(\bm{k}) = (-1)^{\mathrm{dep}(\bm{k})} \sum_{l \geq 0} \left(\sum_{\begin{subarray}{c} \bm{l} \in (\mathbb{Z}_{\geq 0})^{\mathrm{dep}(\bm{k})} \\ \mathrm{wt}(\bm{l}) = l \end{subarray}}\sum_{\bm{l} \oplus \bm{k} \preceq \bm{m} \preceq \bm{l} \oslash \bm{k}} \overline{\widehat{\xi}(\bm{m})} \right)t^{l}
     \]
     and
     \[
        \zeta_{\Shat}(\bm{k}) = (-1)^{\mathrm{dep}(\bm{k})} \sum_{l \geq 0} \left(\sum_{\begin{subarray}{c} \bm{l} \in (\mathbb{Z}_{\geq 0})^{\mathrm{dep}(\bm{k})} \\ \mathrm{wt}(\bm{l}) = l \end{subarray}}\sum_{\bm{l} \oplus \bm{k} \preceq \bm{m} \preceq \bm{l} \oslash \bm{k}} \zeta_{\Shat}(\bm{m}) \right)t^{l}.
     \]
\end{cor}


\end{document}